\documentclass[twoside,leqno,10pt, A4]{amsart}
\usepackage{amsfonts}
\usepackage{amsmath}
\usepackage{amscd}
\usepackage{amssymb}
\usepackage{amsthm}
\usepackage{amsrefs}
\usepackage{latexsym}
\usepackage{mathrsfs}
\usepackage{bbm}
\usepackage{amscd}
\usepackage{amssymb}
\usepackage{amsthm}
\usepackage{amsrefs}
\usepackage{latexsym}
\usepackage{mathrsfs}
\usepackage{bbm}
\usepackage{enumerate}
\usepackage{graphicx}
\allowdisplaybreaks
\usepackage{color}
\begin{document}

\newtheorem{theorem}[subsection]{Theorem}
\newtheorem{proposition}[subsection]{Proposition}
\newtheorem{lemma}[subsection]{Lemma}
\newtheorem{corollary}[subsection]{Corollary}
\newtheorem{conjecture}[subsection]{Conjecture}
\newtheorem{prop}[subsection]{Proposition}
\newtheorem{defin}[subsection]{Definition}

\numberwithin{equation}{section}
\newcommand{\mr}{\ensuremath{\mathbb R}}
\newcommand{\mc}{\ensuremath{\mathbb C}}
\newcommand{\N}{\mathbb{N}}
\newcommand{\dif}{\mathrm{d}}
\newcommand{\intz}{\mathbb{Z}}
\newcommand{\ratq}{\mathbb{Q}}
\newcommand{\natn}{\mathbb{N}}
\newcommand{\comc}{\mathbb{C}}
\newcommand{\rear}{\mathbb{R}}
\newcommand{\prip}{\mathbb{P}}
\newcommand{\uph}{\mathbb{H}}
\newcommand{\fief}{\mathbb{F}}
\newcommand{\majorarc}{\mathfrak{M}}
\newcommand{\minorarc}{\mathfrak{m}}
\newcommand{\sings}{\mathfrak{S}}
\newcommand{\fA}{\ensuremath{\mathfrak A}}
\newcommand{\mn}{\ensuremath{\mathbb N}}
\newcommand{\mq}{\ensuremath{\mathbb Q}}
\newcommand{\half}{\tfrac{1}{2}}
\newcommand{\f}{f\times \chi}
\newcommand{\summ}{\mathop{{\sum}^{\star}}}
\newcommand{\chiq}{\chi \bmod q}
\newcommand{\chidb}{\chi \bmod db}
\newcommand{\chid}{\chi \bmod d}
\newcommand{\sym}{\text{sym}^2}
\newcommand{\hhalf}{\tfrac{1}{2}}
\newcommand{\sumstar}{\sideset{}{^*}\sum}
\newcommand{\sumprime}{\sideset{}{'}\sum}
\newcommand{\sumprimeprime}{\sideset{}{''}\sum}
\newcommand{\sumflat}{\sideset{}{^\flat}\sum}
\newcommand{\shortmod}{\ensuremath{\negthickspace \negthickspace \negthickspace \pmod}}
\newcommand{\V}{V\left(\frac{nm}{q^2}\right)}
\newcommand{\sumi}{\mathop{{\sum}^{\dagger}}}
\newcommand{\mz}{\ensuremath{\mathbb Z}}
\newcommand{\leg}[2]{\left(\frac{#1}{#2}\right)}
\newcommand{\muK}{\mu_{\omega}}
\newcommand{\thalf}{\tfrac12}
\newcommand{\lp}{\left(}
\newcommand{\rp}{\right)}
\newcommand{\Lam}{\Lambda_{[i]}}
\newcommand{\lam}{\lambda}
\newcommand{\af}{\mathfrak{a}}
\newcommand{\sw}{S_{[i]}(X,Y;\Phi,\Psi)}
\newcommand{\lz}{\left(}
\newcommand{\pz}{\right)}
\newcommand{\bfrac}[2]{\lz\frac{#1}{#2}\pz}
\newcommand{\odd}{\mathrm{\ primary}}
\newcommand{\even}{\text{ even}}
\newcommand{\res}{\mathrm{Res}}
\newcommand{\sumn}{\sumstar_{(c,1+i)=1}  w\left( \frac {N(c)}X \right)}
\newcommand{\lab}{\left|}
\newcommand{\rab}{\right|}
\newcommand{\Go}{\Gamma_{o}}
\newcommand{\Ge}{\Gamma_{e}}
\newcommand{\M}{\widehat}
\def\su#1{\sum_{\substack{#1}}}
\newcommand{\Echar}{\mathbb{E}^{\text{char}}}
\newcommand{\E}{\mathbb{E}}
\newcommand{\p}{\mathbb{P}}

\theoremstyle{plain}
\newtheorem{conj}{Conjecture}
\newtheorem{remark}[subsection]{Remark}

\newcommand{\pfrac}[2]{\left(\frac{#1}{#2}\right)}
\newcommand{\pmfrac}[2]{\left(\mfrac{#1}{#2}\right)}
\newcommand{\ptfrac}[2]{\left(\tfrac{#1}{#2}\right)}
\newcommand{\pMatrix}[4]{\left(\begin{matrix}#1 & #2 \\ #3 & #4\end{matrix}\right)}
\newcommand{\ppMatrix}[4]{\left(\!\pMatrix{#1}{#2}{#3}{#4}\!\right)}
\renewcommand{\pmatrix}[4]{\left(\begin{smallmatrix}#1 & #2 \\ #3 & #4\end{smallmatrix}\right)}
\def\en{{\mathbf{\,e}}_n}

\newcommand{\ppmod}[1]{\hspace{-0.15cm}\pmod{#1}}
\newcommand{\ccom}[1]{{\color{red}{Chantal: #1}} }
\newcommand{\acom}[1]{{\color{blue}{Alia: #1}} }
\newcommand{\alexcom}[1]{{\color{green}{Alex: #1}} }
\newcommand{\hcom}[1]{{\color{brown}{Hua: #1}} }

\makeatletter
\def\widebreve{\mathpalette\wide@breve}
\def\wide@breve#1#2{\sbox\z@{$#1#2$}%
     \mathop{\vbox{\m@th\ialign{##\crcr
\kern0.08em\brevefill#1{0.8\wd\z@}\crcr\noalign{\nointerlineskip}%
                    $\hss#1#2\hss$\crcr}}}\limits}
\def\brevefill#1#2{$\m@th\sbox\tw@{$#1($}%
  \hss\resizebox{#2}{\wd\tw@}{\rotatebox[origin=c]{90}{\upshape(}}\hss$}
\makeatletter

\title[On the Gauss circle problem over smooth numbers]{On the Gauss circle problem over smooth numbers}

\author[P. Gao]{Peng Gao}
\address{School of Mathematical Sciences, Beihang University, Beijing 100191, China}
\email{penggao@buaa.edu.cn}

\begin{abstract}
   The Gauss circle problem concerns with the evaluation of $\sum_{n \leq x}r(n)$, where $r(n)$ denotes the number of representations of $n$ as sums of two squares and $x \geq 2$. Let $\Psi_G(x,y)$ denote the sum of $y$-smooth numbers below $x$ weighted by $r(n)$. In this paper, we evaluate $\Psi_G(x,y)$  asymptotically for certain ranges of $x \geq y \geq 2$.
\end{abstract}

\maketitle



\section{Introduction}\label{sec1}

 Let $r(n)$ denote the number of representations of $n$ as sums of two squares of integers, where different signs and different orders of the summands are counted as different. The well-known Gauss circle problem can be viewed as the evaluation of the sum $\sum_{n \leq x}r(n)$ for real $x \geq 2$.

  It is shown by Gauss that (see \cite[(1.70)]{iwakow}) for $\theta=1/2$, we have
\begin{align}
\label{rnaverage}
  \sum_{n \leq x}r(n)=\pi x+O(x^{\theta}).
\end{align}  
   The best estimate for the error term in \eqref{rnaverage} up to date is due to M. N. Huxley, who proved in \cite{Huxley03} that the above expression holds with $\theta=131/208+\varepsilon$ for any $\varepsilon>0$. We remark here (see \cite[p. 215]{iwakow}) that it is conjectured that the best possible error term in \eqref{rnaverage} should be $\theta=1/4+\varepsilon$ for any $\varepsilon>0$. 

  We denote $P(n)$ the largest prime factor of any positive integer $n$. We say an integer $n \geq 1$ is $y$-smooth or $y$-friable if $P(n) \leq y$. 
  In \cite{Goswami24}, A. Goswami studied the Gauss circle problem over smooth integers. To be more precise, we define for $x \geq y \geq 2$, 
\begin{align}
\label{Psif}
 \Psi_G(x,y):=\sum_{\substack{n \leq x \\ P(n) \leq y}}r(n). 
\end{align}
   Let $\rho(u)$, known as the Dickman function to be the unique continuous function on $[0, \infty)$ satisfying the differential difference equation
\begin{align*}
 u\rho'(u)+\rho(u-1)=0,
\end{align*}
  subject to the initial condition $\rho(u)=1, 0 \leq u \leq 1$. 
  
  We further set $x=y^u$ for any $x \geq y \geq 2$ throughout the paper so that $u \geq 1$. Then it is shown in \cite[Theorem 2.2]{Goswami24} that for any $\varepsilon>0$ and $\exp((\log x)^{2/3+\varepsilon}) \leq y \leq x$, there exists a constant $c>0$ such that uniformly in $u$ and $y$, 
\begin{align}
\label{psiGasymp}
 \Psi_G(x,y)=\pi\rho(u)x+O\Big(\frac {x\rho(u)\log^2(u+1)}{\log y}+xu^2\log u  \log y \exp(-c\sqrt{\log y})\Big ).
\end{align}  
  
   It is the aim of this paper to evaluate $\Psi_G(x,y)$ asymptotically for a wider range of $u$. Our approach here is inspired by the treatments employed by
A. Hildebrand and G. Tenenbaum \cite{HT86} to obtain an asymptotic formula for  $\Psi(x,y)$, where $\Psi(x,y)=\sum_{\substack{n \leq x \\ P(n) \leq y}}1$ denotes the number of integers $\leq x$ that are $y$-smooth. Their approach starts by observing that one may apply Rankin's trick to see that for any $\sigma > 0$, $x \geq y \geq 2$, 
\begin{align*}
\Psi(x,y) = \sum_{\substack{n \leq x \\ P(n) \leq y}} 1 \leq \sum_{P(n) \leq y} \biggl( \frac{x}{n} \biggr)^{\sigma} =  x^{\sigma} \zeta(\sigma, y), 
\end{align*}
   where for complex $s=\sigma+it$ with $\sigma, t \in \mr$, we set 
\begin{align*}
\zeta(s, y)=\prod_{p \leq y}(1-p^{-s})^{-1}.
\end{align*}
  Here and throughout the paper, we reserve the letter $p$ for a prime number.
 
  We write $\varphi(s,y)=\log \zeta(s, y)$ and we denote $\varphi_k(s,y)$ the $k$-th partial derivative of $\varphi(s,y)$ with respect to $s$ for all integers $k \geq 1$. For $x \geq y \geq 2$, let $\alpha=\alpha(x,y)$ denote the unique solution of the equation $\varphi_1(\alpha, y)+\log x=0$, then it is shown in \cite{HT86} that $x^{\sigma} \zeta(\sigma, y)$ is minimized at this $\alpha$. The key idea in \cite{HT86} is to evaluate $\Psi(x,y)$ using Perron's formula by choosing the integral involved (say with respect to the variable $s$) to be on the line $\Re(s)=\alpha$. This way, it was shown in \cite[Theorem 1]{HT86} that uniformly for $x \geq y \geq 2$,
\begin{align*}
\Psi(x, y)=\frac {x^{\alpha}\zeta(\alpha,y)}{\alpha\sqrt{2\pi \varphi_2(\alpha,y)}}\Big (1+ O(\frac 1u+\frac {\log y}{y}) \Big).
\end{align*}
 Moreover, asymptotical behaviors of both $\alpha(x,y)$ and $\varphi_2(\alpha,y)$ were given in \cite[Theorem 2]{HT86}.

  In our case, we denote for $s \in \mc$, 
\begin{align*}
  H(s,G;y) :=\sum_{\substack{n \\ P(n) \leq y}}\frac {r(n)/4}{n^s}. 
\end{align*}
   Here we note that we consider $r(n)/4$ instead of $r(n)$ is because that the function $r(n)$ is multiplicative (see the remark made below \eqref{zetaKdecomp}). We  set $\phi(s,G;y) :=\log H(s,G;y)$, and we denote $\phi_k(s,G;y)$ the $k$-th partial derivative of $\phi(s,G;y)$ with respect to $s$ for all integers $k \geq 1$. 
  
  We apply Rankin's trick to see that for any $\sigma > 0$, $x \geq y \geq 2$, 
\begin{equation*}
\Psi_G(x,y) = \sum_{\substack{n \leq x \\ P(n) \leq y}} r(n) \leq \sum_{P(n) \leq y} r(n)\biggl( \frac{x}{n} \biggr)^{\sigma} =  4x^{\sigma} H(\sigma,G; y) .
\end{equation*}
  It follows that 
\begin{align}
\label{Psibound}
  \Psi_G(x,y) \leq 4\min_{\sigma>0}x^{\sigma} H(\sigma,G; y).
\end{align}
  It is shown in Lemma \ref{lemalphafunique} below that the infimum above is in fact a minimum, and is attained for a unique $\alpha_G=\alpha_G(x,y)>0$ defined as
the solution of the equation 
\begin{align}
\label{alphafdef}
  \log x+\phi_1(\alpha_G, G; y)=0.
\end{align}
  We shall henceforth call $\alpha_G$ the saddle point. 
  
Our main result evaluates $\Psi_G(x, y)$ in terms of $H(\alpha_G,G;y)$. 
\begin{theorem}
\label{thm: psiasymp}
 With the notation as above. Let $x \geq y \geq 2$ and set $x=y^u$. For any fixed $\varepsilon_0>0$, we have for $(\log \log y)^2 \leq u \leq y^{1/(2+\varepsilon_0)}/\log y$ and $y$ large enough,
\begin{align}
\label{psiasymp}
\Psi_G(x, y)=\frac {4x^{\alpha_G}H(\alpha_G,G;y)}{\alpha_G\sqrt{2\pi \phi_2(\alpha_G,G;y)}}\Big (1+ O(\frac 1u) \Big ).
\end{align}
\end{theorem} 
   
  The proof of Theorem \ref{thm: psiasymp} is similar to that of the proof of \cite[Theorem 1]{HT86}. We also apply Perron's formula to evaluates $\Psi_G(x, y)$, and again with the integral involved to be on the line determined by the saddle point. 

  For $u>1$, we define $\xi=\xi(u)$ to be the non-zero solution of the equation
\begin{align*}
\begin{split}
 e^{\xi}=1+u\xi.
\end{split}
\end{align*}   
   We also set $\xi(1)=0$. 

  Our next result evaluates the main term on the right-hand side of \eqref{psiasymp} asymptotically in terms of simple functions of $x, y$.  
\begin{theorem}
\label{thm: alphaasymp}
 With the notation as above. Let $x \geq y \geq 2$ and set $x=y^u$. For any fixed $\varepsilon_0>0$, we have uniformly for $(\log x)^{2+\varepsilon_0} <y<x$,
\begin{align}
\label{alphaasymp}
\begin{split}
\alpha_G(x, y)=& \frac {\log (1+y/\log x)}{\log y}\Big (1+ O\big(\frac {\log \log (1+y)}{\log y}\big) \Big ), \\
\phi_2(x,G; y)=&(1+\frac {\log x}{y})\log x \cdot \log y \Big (1+ O\big(\frac 1{\log (1+u)}+\frac 1{\log y}\big) \Big ), \\
\frac {4x^{\alpha_G}H(\alpha_G,G;y)}{\alpha_G\sqrt{2\pi \phi_2(\alpha_G,G; y)}} 
=& \pi x\Big (\frac {\xi'(u)}{2\pi}\Big )^{1/2}\exp \Big (-\gamma_0 -u\xi(u)+\int^{\xi(u)}_0\frac {e^s-1}{s}ds+ O_{\varepsilon}\big(\frac {\log (1+u)}{\log y}+u\exp(-(\log y)^{3/5-\varepsilon})\big) \Big ),
\end{split}
\end{align}
 where $\gamma_0=0.577\ldots$ is the Euler constant. 
\end{theorem}  
 
   Note that it is shown in \cite[Corollary 2]{HT86} that as $u \rightarrow \infty$, we have
\begin{align*}
\begin{split}
\rho(u)=(1+O(\frac 1u))\Big (\frac {\xi'(u)}{2\pi}\Big )^{1/2}\exp \Big (-\gamma_0 -u\xi(u)+\int^{\xi(u)}_0\frac {e^s-1}{s}ds\Big ).
\end{split}
\end{align*}   
  Thus the result given in \eqref{alphaasymp} is consistent with that given in \eqref{psiGasymp}. Moreover, Theorem \ref{thm: psiasymp} and Theorem \ref{thm: alphaasymp} and the above together implies an asymptotical formula for $\alpha_G(x, y)$ that is similar to that given in \eqref{psiGasymp} but valid for much larger values of $u$. 
 
  The proof of Theorem \ref{thm: alphaasymp} is similar to that of \cite[Theorem 2]{HT86} upon using \eqref{merten4} below, so we omit it here. We only point out that in the proof, we may assume $y$ is large enough so that by \eqref{equ: Saddle point approx in intro1} below, the condition $u \leq y^{1/(2+\varepsilon_0)}/\log y$ given in Theorem \ref{thm: psiasymp} or the condition $(\log x)^{2+\varepsilon_0} <y$ given in Theorem \ref{thm: alphaasymp} implies that we have 
\begin{align}
\label{alphalowerbound}
\begin{split}
\alpha_G \geq \frac 12+\varepsilon_1, 
\end{split}
\end{align}
  where $\varepsilon_1>0$ is a constant depending on $\varepsilon_0$ only. This in turn implies that the contribution from sums over $p^{-2\alpha_G}$ is bounded and therefore may be ignored.

\section{Preliminaries}
\label{sec 2}

  We now collect some auxiliary results needed in our proof of Theorem \ref{thm: psiasymp}.

\subsection{Sums over primes}
\label{sec:sum primes}

   Recall that we denote $\chi_4$ the (only) primitive Dirichlet character modulo $4$, so that we have $\chi_4(n) =(-1)^{\frac {n-1}2}$ if $(n,2)=1$, $\chi_4(n)=0$ if $2|n$. 
    We include in this section a result on certain sums over primes.
\begin{lemma}
\label{RS}
 With the notation as above. Let $x \geq 2$ and let $C>0$ be a constant. We have, for some constants $\gamma_G$ and $b_i>0, 1 \leq i \leq 4$, and for $0 \leq \sigma \leq 1+C/\log x$, 
\begin{align}
\label{merten2-0}
\sum_{p\le x} \log p =& 
  x+O\Big(x \exp(-b_1\sqrt{\log x})\Big), \\
\label{merten2}
\sum_{p\le x} \chi_4(p)\log p =&  
  O\Big(x \exp(-b_2\sqrt{\log x})\Big),  \\
\label{sumpsigmalogp}
\sum_{p\le x} \frac{\log p}{p^{\sigma}} =& \int^{x}_1\frac {du}{u^{\sigma}}+ O\Big(1+x^{1-\sigma}\exp(-b_3\sqrt{\log x})\Big), \\
\label{sumpsigmalogpchi}
\sum_{p\le x} \frac{\chi_4(p)\log p}{p^{\sigma}} =& O\Big(1+x^{1-\sigma}\exp(-b_4\sqrt{\log x})\Big), \\
\label{merten4}
 \prod_{p\le x}(1-\frac {1}{p})^{-1}(1-\frac {\chi_4(p)}{p})^{-1}=& \frac {\pi}{4}e^{\gamma_0}\log x+O(1).
\end{align}
\end{lemma}
\begin{proof}
  The expression in \eqref{merten2-0} follows from \cite[Theorem 6.9]{MVa1}. To establish \eqref{merten2}, we let $K=\mq(i)$ be the Gaussian number field and consider the Dedekind zeta function $\zeta_K(s)$ of $K$, which is defined for $\Re(s)>1$ to be
\begin{align*}
\zeta_K(s)=\sum_{\mathcal{A} \subset
  \mathcal{O}_K}\frac {1}{N(\mathcal{A})^{s}},
\end{align*}   
where $\mathcal{A}$ runs over all non-zero integral ideals in $K$ and $N(\mathcal{A})$ is the norm of $\mathcal{A}$. 

  It is shown in \cite[p. 17]{iwakow} that we have
\begin{align}
\label{zetaKdecomp}
\zeta_K(s)=\zeta(s)L(s,\chi_4)=\sum_{n \geq 1}\frac {r(n)}{4n^s}. 
\end{align}   
  In particular, this implies that the function $r(n)/4$ is multiplicative. 
  Note also that it is established in \cite{Sokolovskii68} a Vinogradov--Korobov type zero-free region $\zeta_K(s)$. Namely, for $s=\sigma+it$ with $\sigma, t \in \mr$, then for suitable positive constants $c$ and $t_0$, $\zeta_K(s) \neq 0$ in the region
\begin{align}
\label{zetaKzerofree}
 \sigma \geq 1-c(\log |t|)^{-2/3}(\log \log |t|)^{-1/3}, \quad |t| \geq t_0. 
\end{align}  
   In view of \eqref{zetaKdecomp}, the above now implies that both $\zeta(s)$ and $L(s,\chi_4)$ are non-zero in the above region. In particular, there is no exceptional zero for $L(s,\chi_4)$ so that the estimation \eqref{merten2} follows from \cite[(11.24)]{MVa1}. 
  
  We apply \eqref{merten2-0} and partial summation to see that
\begin{align}
\label{sumpsigmalogp1}
\sum_{p\le x} \frac{\log p}{p^{\sigma}} =& \int^x_2\frac {du}{u^{\sigma}}+ \int^x_2\frac {dO\Big(x \exp(-b_1\sqrt{\log x})\Big)}{u^{\sigma}}.
\end{align}  
  It follows from the proof of \cite[Lemma 7.3]{MVa1} that for some constant $b_3>0$, the last integration above is
\begin{align}
\label{sumpsigmalogp2}
 \ll 1+x^{1-\sigma}\exp(-b_3\sqrt{\log x}).
\end{align} 
  Here we note that \cite[Lemma 7.3]{MVa1} is stated for $0 \leq \sigma \leq 1$. However, an inspection of the proof, especially that lead to \cite[(7.19)]{MVa1} shows that \eqref{sumpsigmalogp2} continues to hold for $0 \leq \sigma \leq 1+C/\log x$. 
  Note moreover that we have $ \int^x_2\frac {du}{u^{\sigma}}=\int^x_1\frac {du}{u^{\sigma}}+O(1)$.  Substituting this and \eqref{sumpsigmalogp2} into \eqref{sumpsigmalogp1} now leads to the validity of \eqref{sumpsigmalogp}. The estimation in \eqref{sumpsigmalogpchi} is similarly established upon using \eqref{merten2}. 

  Lastly, to establish \eqref{merten4}, we may assume that $x$ is sufficiently large. We then note that by Mertens' formula (see \cite[Theorem 2.7 (e)]{MVa1}), we have for $x \geq 2$, 
\begin{align}
\label{prodp}
\prod_{p\le x} (1-\frac {1}{p})^{-1}=& e^{\gamma_0}\log x + O(1).
\end{align}  

   We recall that the residue of $\zeta_K(s)$ at $s=1$ equals $\pi/4$. In view of \eqref{zetaKdecomp} and keep in mind that the residue of $\zeta(s)$ at $s=1$ equals $1$, we deduce that $L(1, \chi_4)=\pi/4$. Now we have
\begin{align}
\label{prodchip}
\prod_{p\le x} (1-\frac {\chi_4(p)}{p})^{-1}=& \prod_{p} (1-\frac {\chi_4(p)}{p})^{-1}\prod_{p> x} (1-\frac {\chi_4(p)}{p})^{-1}=L(1, \chi_4(p))\prod_{p> x} (1-\frac {\chi_4(p)}{p})^{-1}=\frac {\pi}{4}\prod_{p> x} (1-\frac {\chi_4(p)}{p})^{-1}.
\end{align}   
  Note that we have
\begin{align*}
\prod_{p> x} (1-\frac {\chi_4(p)}{p})^{-1}=\exp \Big (\sum_{p>x}-\log (1-\frac {\chi_4(p)}{p} )\Big )=\exp \Big (\sum_{p>x}\frac {\chi_4(p)}{p}+O(\sum_{p>x}\frac 1{p^2})\Big )=\exp \Big (\sum_{p>x}\frac {\chi_4(p)}{p}+O(\frac 1{x})\Big ).
\end{align*}     
 We set $\sigma=1$ in \eqref{sumpsigmalogpchi} to see that $\sum_{p\le x} \frac{\chi_4(p)\log p}{p} = O(1)$. Applying this and partial summation, we see that
 $\sum_{p>x}\frac {\chi_4(p)}{p}=O(1/\log x)$. As $e^z=1+O(|z|)$ for complex $|z| \leq 1$, we see from the above that when $x$ is large enough, we have
\begin{align*}
\prod_{p> x} (1-\frac {\chi_4(p)}{p})^{-1}=\exp \Big (O(\frac 1{\log x})\Big )=1+O(\frac 1{\log x}).
\end{align*} 
  Substituting the above into \eqref{prodchip}, we see that
\begin{align*}
\prod_{p\le x} (1-\frac {\chi_4(p)}{p})^{-1}=& \frac {\pi}{4}\big(1+O(\frac 1{\log x})\big).
\end{align*}        
  This together with \eqref{prodp} now implies the validity of \eqref{merten4} and thus completes the proof of the lemma. 
\end{proof}

   We denote $L_p$ the local factor at the prime $p$ in the Euler product of $L$ for any $L$-function. In particular, $\zeta_p$ denotes the local factor at the prime $p$ in the Euler product of the Riemann zeta-function $\zeta(s)$. Recall the definition of $H(s,f;y)$ from \eqref{Psif}. In view of \eqref{zetaKdecomp}, we see that for any $s \in \mc$, we have for $y \geq 2$, 
\begin{align}
\label{Lsf}
  H(s,G;y)=\prod_{p \leq y}\zeta_p(s)L_p(s, \chi_4)= \prod_{p \leq y}(1-p^{-s})^{-1}(1-\chi_4(p)p^{-s})^{-1}. 
\end{align}
  Note that $H(s,G;y)$ is meromorphic in $s$ in the whole complex plane. It never vanishes, and all its poles are located on the line $\sigma=0$.

\section{Properties of the saddle point}
\label{Sec: smooth}

We first show that the minimum in \eqref{Psibound} is attained at a unique point $\alpha_G>0$. 
\begin{lemma}
\label{lemalphafunique}
 With the notation as above. For $x \geq y \geq 2$, the minimum at the right-hand side expression of \eqref{Psibound} is attained at a $\alpha_G>0$ which is the unique solution of the equation \eqref{alphafdef}.  
\end{lemma}
\begin{proof}
  We recall from  \eqref{Psibound} that
\begin{align}
\label{Psibound1}
  \Psi_G(x,y) \leq 4\min_{\sigma>0}x^{\sigma} H(\sigma,G; y)=4\min_{\sigma>0} \exp \big(\sigma \log x+\phi(\sigma,G; y)\big ).
\end{align}
  We derive from the definition of $\phi_2(s,G;y)$ and \eqref{Lsf} that
\begin{align}
\label{phi2}
\begin{split}
  \phi_2(s,G;y)=\sum_{p \leq y}\Big(\frac {p^s(\log p)^2}{(p^s-1)^2}+\frac {\chi_4(p)p^s(\log p)^2}{(p^s-\chi_4(p))^2}\Big )=: \sum_{p \leq y}\frac {2p^{\sigma}(\log p)^2f(p^{\sigma}, \chi_4(p))}{(p^{\sigma}-1)^2(p^{\sigma}-\chi_4(p))^2},
\end{split}
\end{align}   
  where 
\begin{align*}
\begin{split}
 f(z, \chi_4(p))=z^2-(1+\chi_4(p))z+1 \geq (z-1)^2 \geq 0, \quad z \geq 0. 
\end{split}
\end{align*}    
      
  The above and \eqref{phi2} now implies that $\phi_2(\sigma,G;y) \geq 0$ for $\sigma \geq 0$, so that the function $\sigma \mapsto \sigma \log x+\phi(\sigma,G; y)$ is a convex function of $\sigma$, which in turn shows that there is a unique $\alpha_G>0$ minimizing it. This $\alpha_G$ certainly satisfies equation \eqref{alphafdef} and the solution of the equation \eqref{alphafdef} is unique in view of $\phi_2(\sigma,G;y) \geq 0$. The above discussions together with \eqref{Psibound1} now allows us to complete the proof of the lemma. 
\end{proof}
 
  Recall that we set $x=y^u$ so that $u \geq 1$. Our next result establishes certain asymptotical properties of $\alpha_G(x,y)$. 
\begin{prop}
\label{lemsumpsigmalogp}
 With the notation as above. For $y$ sufficiently large, we have
\begin{align}
\label{alphalower}
   \alpha_G(y^u,y)\geq &\frac 2{\log y}, \quad 1 \leq u \leq \frac {y}{8\log y}, \\
\label{alphausmall}
   \alpha_G(y^u,y) = & 1+O(\frac {1}{\log y}),  \quad 1 \leq u \leq 14, \\
\label{alphaupper}
   \alpha_G(y^u,y)\leq & 1-\frac {4}{\log y},  \quad u \geq 14, \\  
\label{equ: Saddle point approx in intro}
\alpha_G(y^u,y) = & 1 - \frac{\xi(u)}{\log y} + O \biggl( \frac{1}{(\log y)^2}+\frac {u}{y}\biggr), \quad 1 \leq u \leq \frac {y}{8\log y}, \\
\label{equ: Saddle point approx in intro1}
  \alpha_G(y^u,y)=&  1 - \frac{\log(u\log u)}{\log y} + O \biggl( \frac{\log\log u}{\log u\log y}+\frac{1}{(\log y)^2}+\frac {u}{y}\biggr), \quad 3 \leq u \leq \frac {y}{8\log y}.
\end{align}
\end{prop}
\begin{proof}
  We deduce from \eqref{alphafdef} and \eqref{Lsf} that we have for any $x \geq y \geq 2$, 
\begin{align}
\label{alpharel}
\begin{split}
 \log x =\sum_{p \leq y} \Big(\frac {\log p}{p^{\alpha_G (x,y)}-1}+\frac {\chi_4(p)\log p}{p^{\alpha_G  (x,y)}-\chi_4(p)}\Big) =:\sum_{p \leq y} f_1(p^{\alpha_G  (x,y)}  (x,y), \chi_4(p))\log p, 
\end{split}
\end{align}     
  where 
\begin{align*}
\begin{split}
 f_1(z, \chi_4(p))=\frac {1}{z-1}+\frac {\chi_4(p)}{z-\chi_4(p)}=\frac {1}{z-1}-1+\frac {z}{z-\chi_4(p)}, \quad z>1.
\end{split}
\end{align*}   
  
  It follows from the above that $f_1(z, \chi_4(p))$ is an increasing function of $\chi_4(p)$ for any fixed $z > 1$.  As $\chi_4(p)$ only takes value $\pm 1$ or $0$, we deduce that for $z>1$, 
\begin{align}
\label{g2est}
\begin{split}
 f_1(z, \chi_4(p)) \leq & \frac {1}{z-1}+\frac {1}{z-1}= \frac 2{z-1}, \\
 f_1(z, \chi_4(p)) \geq &  \frac {1}{z-1}+\frac {1}{z+1}=\frac {2}{z^2-1}>0.
\end{split}
\end{align} 
  Moreover, we notice that $\chi_4(p)=1$ when $p \equiv 1 \pmod 4$ so that we have
\begin{align}
\label{f1p14}
\begin{split}
 f_1(z, \chi_4(p)) = & \frac {1}{z-1}+\frac {1}{z-1}= \frac 2{z-1}, \quad  p \equiv 1 \pmod 4.
\end{split}
\end{align}   
  We apply from the second estimation in \eqref{g2est} and \eqref{f1p14} in \eqref{alpharel} to see that  
\begin{align}
\label{ulogylower}
u\log y=\sum_{p \leq y}\Big(\frac {\log p}{p^{\alpha_G}-1}+\frac {\chi_4(p)\log p}{p^{\alpha_G}-\chi_4(p)}\Big ) \geq 2\sum_{\substack{p \leq y \\ p \equiv 1 \pmod 4}} \frac {\log p}{p^{\alpha_G}-1} \geq \frac {2}{y^{\alpha_G}-1}\sum_{\substack{p \leq y \\ p \equiv 1 \pmod 4}} \log p .
\end{align}  
   
  We then use \eqref{merten2-0} and  \eqref{merten2} to see that for $y$ sufficiently large,  
\begin{align}
\label{ulogylowersimplified}
\sum_{\substack{p \leq y \\ p \equiv 1 \pmod 4}} \log p =\sum_{\substack{p \leq y }} \frac {(1+\chi_4(p))\log p}{2} \geq \frac {2y}{5}.
\end{align}  

  We derive from \eqref{ulogylower} and \eqref{ulogylowersimplified} that
\begin{align*}
 u\log y \geq \frac {4y}{5(y^{\alpha_G}-1)}.
\end{align*} 
  As $u \leq y/(8\log y)$, the above implies that  
\begin{align*}
 \alpha_G \geq \frac {\log (1+4y/(5u\log y))}{\log y} \geq \frac {\log (1+32/5)}{\log y} > \frac {2}{\log y}. 
\end{align*}   
  This establishes the estimation given in \eqref{alphalower}.

  We now assume that $\sigma  \geq \frac {2}{\log y}$ and observe that uniformly for $|\delta| \leq 1/2$, we have $(1-\delta)^{-1}=1+O(\delta)$. By setting $\delta=p^{-\sigma}, v=v(\sigma)=e^{1/\sigma}$, we see that
\begin{align}
\label{sumpvy}
\begin{split}
  \sum_{v<p \leq y}\Big(\frac {\log p}{p^{\sigma}-1}+\frac {\chi_4(p)\log p}{p^{\sigma}-\chi_4(p)}\Big )=\sum_{v<p \leq y}\frac {(1+\chi_4(p)) \log p}{p^{\sigma}}+O(\sum_{v<p \leq y}\frac {\log p}{p^{2\sigma}}).
\end{split}
\end{align}   
  We apply \eqref{sumpsigmalogp} and \eqref{sumpsigmalogpchi} to see that when $\sigma \leq 1+C/\log y$ for any fixed constant $C>0$, then there exists some constant $b_5>0$, 
\begin{align}
\label{sumpvy1}
\begin{split}
\sum_{v<p \leq y}\frac {(1+\chi_4(p)) \log p}{p^{\sigma}} =& \int^{y}_v\frac {dw}{w^{\sigma}}+ O\Big(1+y^{1-\sigma}\exp(-b_5\sqrt{\log y}+v^{1-\sigma}\exp(-b_5\sqrt{\log v})\Big) \\
=& \int^{y}_1\frac {dw}{w^{\sigma}}-\int^{v}_1\frac {dw}{w^{\sigma}}+ O\Big(1+y^{1-\sigma}\exp(-b_5\sqrt{\log y})+v^{1-\sigma}\exp(-b_5\sqrt{\log v})\Big) \\
=& \int^{y}_1\frac {dw}{w^{\sigma}}+ O\Big(v+y^{1-\sigma}\exp(-b_5\sqrt{\log y})\Big).
\end{split}
\end{align}
  
 We next consider
\begin{align*}
\sum_{p \leq v}\Big(\frac {\log p}{p^{\sigma}-1}+\frac {\chi_4(p)\log p}{p^{\sigma}-\chi_4(p)}\Big)=\sum_{p \leq v} f_1(p^{\sigma}, \chi_4(p))\log p.
\end{align*}  

  It follows from \eqref{g2est} that
\begin{align*}
\sum_{p \leq v}\Big(\frac {\log p}{p^{\sigma}-1}+\frac {\chi_4(p)\log p}{p^{\sigma}-\chi_4(p)}\Big) \ll  \sum_{p \leq v} \frac {\log p}{p^{\sigma}-1}=\sum_{p \leq v} \frac {\log p}{p^{\sigma}(1-p^{-\sigma})}.
\end{align*}   
 If $\sigma>1/3$, then the sum above is $\ll 1$. We may thus assume that $2/\log y \leq \sigma \leq 1/3$, in which case we note that it is shown in the proof of \cite[Lemma 7.5]{MVa1} that $(1-p^{-\sigma})^{-1} \ll \log v/\log p$ for $p<v$. We deduce from this and the above that
\begin{align*}
\sum_{p \leq v}\Big(\frac {\log p}{p^{\sigma}-1}+\frac {\chi_4(p)\log p}{p^{\sigma}-\chi_4(p)}\Big) \ll  \sum_{p \leq v} \frac {\log v}{p^{\sigma}} \ll v\log v.
\end{align*} 
  We deduce from this, \eqref{sumpvy} and \eqref{sumpvy1} that when $2/\log y \leq \sigma  \leq 1+C/\log y$ for any fixed constant $C>0$,
\begin{align}
\label{sumpvy3}
\begin{split}
 -\phi_1(\sigma, G;y)= & \sum_{p \leq y}\Big(\frac {\log p}{p^{\sigma}-1}+\frac {\chi_4(p)\log p}{p^{\sigma}-\chi_4(p)}\Big) \\
  =& \int^{y}_1\frac {dw}{w^{\sigma}}+ O\Big(v\log v+y^{1-\sigma}\exp(-b_5\sqrt{\log y})\Big)+O(\sum_{v<p \leq y}\frac {\log p}{p^{2\sigma}}) \\
  =& \int^{y}_1\frac {dw}{w^{\sigma}}+ O\Big(v(\sigma)\log v(\sigma)+y^{1-\sigma}\exp(-b_5\sqrt{\log y})\Big)+O(\sum_{p \leq y}\frac {\log p}{p^{2\sigma}}).
\end{split}
\end{align} 

 Note that in the proof of Lemma \ref{lemalphafunique}, it is shown that $\phi_2(\sigma, G;y) \geq 0$ for $\sigma \geq 0$. This implies that $-\phi_1(\sigma, G;y)$  is a decreasing function of $\sigma>0$. Thus, if $\alpha_G \geq 1+\frac {\delta_1}{\log y}$ for some constant $\delta_1>0$, then we have by \eqref{alphafdef} and \eqref{sumpvy3},   
\begin{align*}
\begin{split}
  u\log y =& -\phi_1(\sigma_G, G;y) \\
  \leq & -\phi_1(1+\frac {\delta_1}{\log y}, G;y) \\
  \leq & \int^{y}_1\frac {dw}{w^{1+\frac {\delta_1}{\log y}}}+ O\Big(v(1+\frac {\delta_1}{\log y})\log v(1+\frac {\delta_1}{\log y})+y^{-\frac {\delta_1}{\log y}}\exp(-b_5\sqrt{\log y})\Big)+O(\sum_{p \leq y}\frac {\log p}{p^{2(1+\frac {\delta_1}{\log y})}}) \\
  \leq & \frac {1-e^{-\delta_1}}{\delta_1}\log y+O(1),
\end{split}
\end{align*}
  This leads to a contradiction for $1 \leq u \leq 14$ by taking $\delta_1$ large enough. Thus, we must have $\alpha_G \leq 1+\frac {\delta_1}{\log y}$ for some constant $\delta_1>0$. Similarly, we must have $\alpha_G \geq 1+\frac {\delta_2}{\log y}$ for some constant $\delta_2>0$. This implies the validity of \eqref{alphausmall}. An analogue argument also leads to the validity of \eqref{alphaupper}.

 It remains to establish \eqref{equ: Saddle point approx in intro}. For this, we first note that it suffices to consider the case $u > 14$, as the case $1 \leq u \leq 14$ follows from the observation that we have trivially $0<\xi(u) \ll \log u$ and \eqref{alphausmall}. When $14 < u \leq y/(8\log y)$, we have by \eqref{alphalower} and \eqref{alphaupper} that we have $2/\log y \leq \alpha_G \leq 1-4/\log y$, which implies that $\alpha_G(1-\alpha_G)\log y \geq 2(1-2/\log y) \geq 3/2$ for $y$ large enough. It follows from this that we have $v(\alpha_G) \leq (y^{1-\alpha_G})^{2/3}$, so that we have $v(\alpha_G)\log v(\alpha_G) \leq (y^{1-\alpha_G})^{3/4}$. Now using the well-known estimation
\begin{align}
\label{1xbound}
\begin{split}
 1+x \leq e^x, \quad x \in \mr,
\end{split}
\end{align}
 we see that $(y^{1-\alpha_G})^{3/4} \ll y^{1-\alpha_G}/((1-\alpha_G)\log y)$. We thus derive that
\begin{align}
\label{vest}
\begin{split}
  v(\alpha_G)\log v(\alpha_G) \ll \frac {y^{1-\alpha_G}}{(1-\alpha_G)\log y}.
\end{split}
\end{align}   

   We next deduce from \eqref{alpharel} and the second estimation in \eqref{g2est} that 
\begin{align}
\label{p2alphaest}
u\log y=\sum_{p \leq y}\Big(\frac {\log p}{p^{\alpha_G}-1}+\frac {\chi_4(p)\log p}{p^{\alpha_G}-\chi_4(p)}\Big ) \geq 2\sum_{\substack{p \leq y }} \frac {\log p}{p^{2\alpha_G}-1} \geq  2\sum_{\substack{p \leq y }} \frac {\log p}{p^{2\alpha_G}} .
\end{align}  
  
   Note that we have $\frac {y^{1-\alpha_G}-1}{1-\alpha_G} \gg 1$ for $2/\log y \leq \alpha_G \leq 1-4/\log y$. Moreover, we deduce from \eqref{alpharel} and \eqref{sumpvy3} that when $2/\log y \leq \alpha_G \leq 1-4/\log y$, 
\begin{align}
\label{sumpvy4}
\begin{split}
  u\log y =\int^{y}_1\frac {dw}{w^{\alpha_G}}+ O\Big(v(\alpha_G)\log v(\alpha_G)+y^{1-\alpha_G}\exp(-b_5\sqrt{\log y})\Big)+O(\sum_{p \leq y}\frac {\log p}{p^{2\alpha_G}}).
\end{split}
\end{align}

   Thus when $u \log y \ll 1$, we deduce from the above, \eqref{vest} and  \eqref{p2alphaest} that 
\begin{align}
\label{ulogylower1}
\begin{split}
 \frac {y^{1-\alpha_G}-1}{1-\alpha_G} \ll  \int^{y}_1\frac {dw}{w^{\alpha_G}}+ O\Big(v(\alpha_G)\log v(\alpha_G)+y^{1-\alpha_G}\exp(-b_5\sqrt{\log y})\Big)+O(\sum_{p \leq y}\frac {\log p}{p^{2\alpha_G}}) \ll u\log y.
\end{split}
\end{align}  

  When $1 \ll u\log y$, we apply \eqref{vest}, \eqref{p2alphaest} into \eqref{sumpvy4} to see that for sufficiently large $y$, 
\begin{align}
\label{ulogylower20}
\begin{split}
 \frac {y^{1-\alpha_G}-1}{1-\alpha_G} \ll  \int^{y}_1\frac {dw}{w^{\alpha_G}}+O\Big(v(\alpha_G)\log v(\alpha_G)+y^{1-\alpha_G}\exp(-b_5\sqrt{\log y})\Big) \ll u\log y+O(\sum_{p \leq y}\frac {\log p}{p^{2\alpha_G}}) \ll u\log y.
\end{split}
\end{align}  
  We thus conclude from \eqref{ulogylower1} and \eqref{ulogylower20} that we always have, for $2/\log y \leq \alpha_G \leq 1-4/\log y$, 
\begin{align}
\label{ypowerupper}
\begin{split}
 \frac {y^{1-\alpha_G}}{1-\alpha_G} \ll \frac {y^{1-\alpha_G}-1}{1-\alpha_G} \ll  u\log y.
\end{split}
\end{align}    
  
  The above implies that 
\begin{align}
\label{ynegalphabound}
\begin{split}
 y^{-\alpha_G} \ll  \frac {(1-\alpha_G)u\log y}{y}.
\end{split} 
\end{align}
  
   We next estimate $\sum_{p \leq y}\frac {\log p}{p^{2\alpha_G}}$ by noting that when $2\alpha_G > 4/3$, we have
\begin{align}
\label{sump2alpha}
\begin{split}
 \sum_{p \leq y}\frac {\log p}{p^{2\alpha_G}} \ll 1.
\end{split}
\end{align}      
  While when $1-2/\log y < 2\alpha_G < 4/3$, we have by \eqref{sumpsigmalogp} that
\begin{align}
\label{sump2alpha1}
\begin{split}
 \sum_{p \leq y}\frac {\log p}{p^{2\alpha_G}} \ll \sum_{p \leq y}\frac {\log p}{p^{1-2/\log y}}\ll \sum_{p \leq y}\frac {\log p}{p} \ll \log y \ll \frac {y^{1-\alpha_G}}{1-\alpha_G} \cdot \frac 1{\log y}.
\end{split}
\end{align}    
  When $2\alpha_G  \leq 1-2/\log y$, we apply \eqref{sumpsigmalogp} and \eqref{ynegalphabound} to see that
\begin{align}
\label{sump2alpha2}
\begin{split}
 \sum_{p \leq y}\frac {\log p}{p^{2\alpha_G}} \ll & \frac {y^{1-2\alpha_G}-1}{1-2\alpha_G}+O(1+y^{1-2\alpha_G}\exp(-b_3\sqrt{\log y})) \\
  \ll &  \frac {y^{1-2\alpha_G}}{1-2\alpha_G}+O(1+y^{1-2\alpha_G}\exp(-b_3\sqrt{\log y})) \\
  \ll & \frac {y^{1-2\alpha_G}}{1-2\alpha_G}+1 \\ 
  \ll & \frac {y^{1-2\alpha_G}}{1-\alpha_G}+1 \\
  \ll & \frac {y^{1-\alpha_G}}{1-\alpha_G} \cdot \frac {u\log y}{y}+1, 
\end{split}
\end{align}

  Putting together \eqref{sump2alpha}--\eqref{sump2alpha2}, we see that
\begin{align}
\label{sump2alpha3}
\begin{split}
 \sum_{p \leq y}\frac {\log p}{p^{2\alpha_G}} \ll & 1+ \frac {y^{1-\alpha_G}}{1-\alpha_G} \cdot \frac {u\log y}{y}+\frac {y^{1-\alpha_G}}{1-\alpha_G} \cdot \frac 1{\log y} \\
  \ll & \frac {y^{1-\alpha_G}}{1-\alpha_G} \cdot \Big ( \frac {(1-\alpha_G)}{y^{1-\alpha_G}}+ \frac {u\log y}{y}+\frac 1{\log y} \Big ) \\
  \ll & \frac {y^{1-\alpha_G}}{1-\alpha_G} \cdot \Big ( \frac 1{\log y}+ \frac {u\log y}{y}\Big ),
\end{split}
\end{align}    
  where the last estimation above follows from \eqref{1xbound}. 

 We now set $u_1=\frac {y^{1-\alpha_G}-1}{(1-\alpha_G)\log y}$ to see from \eqref{vest}, \eqref{sumpvy4} and \eqref{sump2alpha3} that
\begin{align*}
\begin{split}
  \frac {u}{u_1}-1=O\Big (\frac 1{\log y}+ \frac {u\log y}{y}\Big).
\end{split}
\end{align*} 
 The above implies that $u \asymp u_1$ for $y$ large enough. Moreover, it follows from the definition of $\xi(u_1)$ that we have $\xi(u_1)=(1-\alpha_G)\log y$. Note further that we have $\xi'(u) =\frac 1u(1+O(1/\log u))$ by \cite[Lemma 4, ii)]{Hildebrand86}. We deduce from these and the mean value theorem that
\begin{align*}
\begin{split}
  1-\alpha_G-\frac {\xi(u)}{\log y}=\frac {\xi(u_1)-\xi(u)}{\log y} \ll \frac {|u-u_1|}{u_1\log y}=O\Big(\frac 1{(\log y)^2}+ \frac {u}{y} \Big ).
\end{split}
\end{align*}   
  This implies the validity of \eqref{equ: Saddle point approx in intro}. 

  Lastly, the expression in \eqref{equ: Saddle point approx in intro1} follows from \eqref{equ: Saddle point approx in intro}, and \cite[Lemma 1]{HT86}, which shows that for $u \geq 3$,
\begin{align*}
\begin{split}
 \xi(u)=\log (u \log u)+O\Big(\frac {\log \log u}{\log u} \Big ).
\end{split}
\end{align*}   
 This completes the proof of the proposition. 
\end{proof}

\section{Proof of Theorem \ref{thm: psiasymp}}

\subsection{Some lemmas}    
  As a preparation, we include in this section various estimations needed in our proof. We begin with a result concerning the size of $\phi_k(\alpha_G,G;y)$ for $1 \leq k \leq 4$. 
\begin{lemma}
\label{lemphibound}
  With the notation as above. Let $y \geq 2, x=y^u$. We have uniformly for $1 \leq u \leq y/(8\log y), \alpha_G=\alpha_G(x,y) \geq 1/2+\varepsilon_1$ for any fixed $\varepsilon_1>0$ and $1 \leq k \leq 4$,
\begin{align}
\label{phikasymp}
  |\phi_k(\alpha_G,G;y)| \asymp u(\log y)^k. 
\end{align}
\end{lemma}  
\begin{proof}
  Note that we have for $1 \leq k \leq 4$,
\begin{align}
\label{phik}
\begin{split}
  |\phi_k(\alpha_G,G;y)| \asymp & \Big|\sum_{p \leq y}(\frac {p^{(k-1)\alpha_G}(\log p)^k}{(p^{\alpha_G}-1)^k}+\frac {\chi_4(p)p^{(k-1)\alpha_G}(\log p)^k}{(p^{\alpha_G}-\chi_4(p))^k}\Big |.
\end{split}
\end{align}     
   Now as $\alpha_G \geq 1/2+\varepsilon_1$, we have $2\alpha_G \geq 1+2\varepsilon_1$, so that by Taylor expansion, the right-hand side expression above is
\begin{align}
\label{phik1}
\begin{split}
  = & \Big|O(1)+\sum_{p \leq y}\frac {(1+\chi_4(p))(\log p)^k}{p^{\alpha_G}}\Big |. 
\end{split}
\end{align}     
  As $1+\chi_4(p) \geq 0$, we see that the above is
\begin{align}
\label{phik2}
\begin{split}
  \leq \Big |\sum_{p \leq y}\frac {(1+\chi_4(p))(\log p)^k}{p^{\alpha_G}}\Big |+O(1)
  \leq (\log y)^{k-1} \sum_{p \leq y}\frac {(1+\chi_4(p))\log p}{p^{\alpha_G}}+O(1) \ll u(\log y)^{k}, 
\end{split}
\end{align} 
 where the last estimation above follows from \eqref{alpharel} and \eqref{sumpvy} (by setting $v=1$ there), keeping in mind that we have $\alpha_G \geq 1/2+\varepsilon_1$ here.   
  
  We deduce from \eqref{phik}--\eqref{phik2} that
\begin{align}
\label{phikupper}
\begin{split}
  |\phi_k(\alpha_G,G;y)| \ll u(\log y)^{k}.
\end{split}
\end{align} 

  On the other hand, we infer from \eqref{sumpsigmalogp} that there exists a constant $b_6>0$ such that
\begin{align}
\label{phiklower1}
\begin{split}
  & \Big|O(1)+\sum_{p \leq y}\frac {(1+\chi_4(p))(\log p)^k}{p^{\alpha_G}}\Big | \\
\geq & \Big |\sum_{\sqrt{y} < p \leq y}\frac {(1+\chi_4(p))(\log p)^k}{p^{\alpha_G}}\Big |+O(1)
  \geq (\log y)^{k-1} \sum_{\sqrt{y} < p \leq y}\frac {(1+\chi_4(p))\log p}{p^{\alpha_G}}+O(1) \\
  \gg &  \frac {y^{1-\alpha_G}-y^{(1-\alpha_G)/2}}{1-\alpha_G}(\log y)^{k-1} +O((\log y)^{k-1} +(\log y)^{k-1}y^{1-\alpha_G}\exp(-b_6\sqrt{\log y})). 
\end{split}
\end{align}      
   Note that when $14 < u \leq y/(8\log y)$, we have by \eqref{alphaupper} that $\alpha_G \leq 1-4/\log y$. It is then easy to see that we have $\frac {y^{1-\alpha_G}-y^{(1-\alpha_G)/2}}{1-\alpha_G} \gg \frac {y^{1-\alpha_G}-1}{1-\alpha_G} \gg \frac {y^{1-\alpha_G}}{1-\alpha_G} $. On the other hand, we deduce from \eqref{vest}, \eqref{sumpvy4} and our assumption $\alpha_G \geq 1/2+\varepsilon_1$ that 
\begin{align*}
\begin{split}
 u \log y \ll \int^y_1\frac {dw}{w^{\alpha_G}}+O(1) \ll \frac {y^{1-\alpha_G}-1}{1-\alpha_G} \ll \frac {y^{1-\alpha_G}}{1-\alpha_G}. 
\end{split}
\end{align*}      

   We deduce from this and \eqref{phiklower1} that 
\begin{align*}
\begin{split}
  \Big|O(1)+\sum_{p \leq y}\frac {(1+\chi_4(p))(\log p)^k}{p^{\alpha_G}}\Big | 
  \gg &  \frac {y^{1-\alpha_G}}{1-\alpha_G}(\log y)^{k-1}+O((\log y)^{k-1}) \gg u(\log y)^{k}. 
\end{split}
\end{align*}      
    It follows from the above, \eqref{phik} and \eqref{phik1} that 
\begin{align}
\label{phiklower}
\begin{split}
  |\phi_k(\alpha_G,G;y)| \gg u(\log y)^{k}.
\end{split}
\end{align}    
  The above together with \eqref{phikupper} now implies \eqref{phikasymp} for the case $14 < u \leq y/(8\log y)$. 
  
  When $1 \leq u \leq 14$, we repeat the above arguments to see that it suffices to assume that $\alpha_G \geq 1-4/\log y$. Moreover, by \eqref{alphausmall}, we may also assume that $\alpha_G \leq 1+C/\log y$ for some constant $C>0$.  By \eqref{vest} and \eqref{sumpvy4}, we see that 
\begin{align}
\label{wintlower}
\begin{split}
  \int^y_1\frac {dw}{w^{\alpha_G}} \gg u \log y.
\end{split}
\end{align}   

  We now apply \eqref{sumpsigmalogp} and \eqref{sumpsigmalogpchi} to see that when $1-4/\log y \leq \alpha_G \leq 1+C/\log y$,  we have for any constant $0<c<1$, 
\begin{align}
\label{phik5}
\begin{split}
  \sum_{y^c < p \leq y}\frac {(1+\chi_4(p))\log p}{p^{\alpha_G}} =\int^y_{1}\frac {dw}{w^{\alpha_G}}+O(1+\int^{y^c}_1\frac {dw}{w^{\alpha_G}}).
\end{split}
\end{align}   
  
  As $\alpha_G \geq 1-4/\log y$, we have   
\begin{align*}
\begin{split}
 \int^{y^c}_1\frac {dw}{w^{\alpha_G}} \leq e^4\int^{y^c}_1\frac {dw}{w} \leq ce^4 \log y.
\end{split}
\end{align*}     
  Thus upon choosing $c$ small enough, we see from \eqref{wintlower}, \eqref{phik5} and the above that 
\begin{align*}
\begin{split}
  \sum_{y^c < p \leq y}\frac {(1+\chi_4(p))\log p}{p^{\alpha_G}} \gg u\log y. 
\end{split}
\end{align*}   
  We now repeat our arguments above for the case $u > 14$ to see that the estimation given in \eqref{phiklower} is valid for the case $1 \leq u \leq 14$ as well. Together with \eqref{phikupper}, it now implies \eqref{phikasymp} for the case  $1 \leq u \leq 14$. 
  This completes the proof of the lemma. 
\end{proof}

  For any fixed $0<\lambda<1$, we define $$Y(\lambda):=\exp ((\log y)^{3/2-\lambda}).$$ 
Our next result provides estimations on ratios of $|H(s,G;y)/H(\alpha_G, G;y)|$ for $s =\alpha_G(x,y)+it, t \in \mr$. 
\begin{lemma}
\label{RS3}
  With the notation as above. Let $x \geq y \geq \log x \geq 2$, $t \in \mr, s =\alpha_G(x,y)+it, t \in \mr$. Suppose that $y$ is large enough and that $1/2+\varepsilon_1 \leq \alpha_G \leq 1-4/\log y$ for any fixed $\varepsilon_1>0$.
  
  (i). If $|t| \leq 1/\log y$, then we have for some constant $c_0>0$, 
\begin{align}
\label{Lrationtsmall}
  \Big |\frac {H(s,G;y)}{ H(\alpha_G,G;y)} \Big | \leq \exp \Big (-\frac {c_0y}{\log y}\log \big(1+\frac {t^2\phi_2(\alpha,G;y)}{(y/\log y)}\big) \Big ).
\end{align}

  (ii). For any fixed $0<\lambda<1$, we have uniformly for
$y \geq y_0(\lambda)$ where $y_0(\lambda)$ is a sufficiently large absolute constant, and for $1/\log y \leq |t| \leq Y^2(\lambda)$,
\begin{align}
\label{Lrationtlarge}
  \Big |\frac {H(s,G;y)}{ H(\alpha_G,G;y)} \Big | \ll_{\varepsilon} \exp \Big (-\frac {c_0ut^2}{(1-\alpha_G)^2+t^2} \Big ). 
\end{align}  
\end{lemma}
\begin{proof}
  We apply \eqref{Lsf} to see that for $\alpha_G \geq 1/2+\varepsilon_1$, 
\begin{align*}
  |H(s,G;y)|=\Big|\prod_{p \leq y}(1-\frac {1+\chi_4(p)}{p^s})^{-1}\prod_{p \leq y}(1+O(\frac {1}{p^{2\alpha_G}}))\Big | \asymp \Big|\prod_{2<p \leq y}(1-\frac {1+\chi_4(p)}{p^s})^{-1}\Big|. 
\end{align*}  
  As $\chi_4(p)=\pm 1$ for $p \equiv \pm 1 \pmod 4$, we deduce from the above that
\begin{align*}
  \Big |\frac {H(s,G;y)}{ H(\alpha_G,G;y)} \Big | \ll \prod_{\substack{p \leq y \\ p \equiv 1 \pmod 4}}\Big | \frac {1-\frac {2}{p^{\alpha_G}}}{1-\frac {2}{p^{s}}} \Big | = \prod_{\substack{p \leq y \\ p \equiv 1 \pmod 4}}\Big (1+ \frac {4(1-\cos(t \log p))}{p^{\alpha_G}(1-2p^{-\alpha_G})^{2}} \Big )^{-1/2}. 
\end{align*}
  As $\alpha_G \geq 1/2+\varepsilon_1 \gg 1/\log y$, we apply \cite[(3.11), (3.12)]{HT86} to see that when $|t| \leq 1/\log y$,
\begin{align*}
   \frac {4(1-\cos(t \log p))}{p^{\alpha_G}(1-2p^{-\alpha_G})^{2}} \leq \Big (\frac {t \log p}{\alpha_G\log p-\log 2} \Big )^2 
   \ll \Big (\frac {t \log p}{\alpha_G\log p} \Big )^2 \leq c_1,
\end{align*}  
  for some constant $c_1>0$. Thus regardless of the values of $\chi_4(p)$, we  deduce from \cite[(3.13)]{HT86} that we have
\begin{align*}
  \Big |\frac {H(s,G;y)}{ H(\alpha_G,G;y)} \Big | \leq  \prod_{\substack{p \leq y \\ p \equiv 1 \pmod 4}}\exp\Big (-\frac {2(1-\cos(t \log p))}{(1+c_1)p^{\alpha_G}(1-2p^{-\alpha_G})^{2}} \Big )
  \leq & \exp\Big (-\sum_{\substack{p \leq y \\ p \equiv 1 \pmod 4}}\frac {4(t \log p)^2}{\pi^2(1+c_1)p^{\alpha_G}(1-2p^{-\alpha_G})^{2}} \Big ),
\end{align*}  
   where the last estimation above follows from \cite[(3.11)]{HT86}. We use $\alpha_G \geq 1/2+\varepsilon_1$ again to see that for some small constant $c_0>0$, 
\begin{align*}
 & \exp\Big (-\sum_{\substack{p \leq y \\ p \equiv 1 \pmod 4}}\frac {4(t \log p)^2}{\pi^2(1+c_1)p^{\alpha_G}(1-2p^{-\alpha_G})^{2}} \Big ) \\
  \ll &  \exp\Big (-\sum_{\substack{p \leq y \\ p \equiv 1 \pmod 4}}\frac {4(t \log p)^2}{\pi^2(1+c_1)p^{\alpha_G}} \Big ) \\
  =& \exp\Big (-\sum_{\substack{2<p \leq y \\ p \equiv 1 \pmod 4}}\frac {2(1+\chi_4(p))(t \log p)^2}{\pi^2(1+c_1)p^{\alpha_G}} \Big ) \\
   \ll & \exp (-c_0t^2\phi_2(\alpha_G,G;y)),
\end{align*}  
  where the last estimation above follows by arguing similar to those given in the proof of Lemma \ref{lemphibound}. 
Now applying the estimation given in \eqref{1xbound}, we see that the estimation given in \eqref{Lrationtsmall} follows. 
 
  Next, to establish the estimation given in \eqref{Lrationtlarge}, we first note that we have trivially
\begin{align}
\label{Hbound}
 |H(s,G;y)|=\Big|\sum_{\substack{n \\ P(n) \leq y}}\frac {r(n)}{n^s}\Big | \leq \sum_{\substack{n \\ P(n) \leq y}}\frac {r(n)}{n^{\alpha_G}} =H(\alpha_G,G;y). 
\end{align}  

  We then apply \cite[(3.14)]{HT86} by setting $v=(1-\cos (t\log p))/2, t=p^{\alpha_G}/2>1$ there to see that
\begin{align*}
 \Big (1+ \frac {4(1-\cos(t \log p))}{p^{\alpha_G}(1-22p^{-\alpha_G})^{2}} \Big )^{-1/2} \ll \exp \Big (  -\frac {2(1-\cos(t \log p))}{p^{\alpha_G}}\Big ). 
\end{align*}  

  It follows that
\begin{align}
\label{Lratiotlarge}
  \Big |\frac {H(s,G;y)}{ H(\alpha_G,G;y)} \Big | \leq \exp \Big (-\sum_{\substack{p \leq y \\ p \equiv 1 \pmod 4}} \frac {2(1-\cos(t \log p))}{p^{\alpha_G}} \Big ). 
\end{align}  

  To estimate the above, we note that it is shown in \cite[Lemma 6]{HT86} that using the zero-free region given in \eqref{zetaKzerofree} for $\zeta(s)$ and the standard Perron's formula (see \cite[Theorem 5.1]{MVa1}), one has for $y \geq 2, 0<\beta, \varepsilon<1, t \in \mr, |t| \leq Y(\lambda), s=1-\beta+it$,
\begin{align}
\label{Lambdapartialsum}
\begin{split}
  \sum_{n \leq y}\frac {\Lambda(n)}{n^s}
=& \frac {y^{\beta-it}}{\beta-it}+O\Big(\frac 1{\beta}(1+y^{\beta}\exp(-(\log y)^{\varepsilon/2}))\Big ),
\end{split}
\end{align}   
  where $\Lambda(n)$ is the Von Mangoldt function defined to be $\Lambda(n)=\log p$ if $n=p^k$ for integers $k \geq 1$ and $\Lambda(n)=0$ otherwise. 
  
  Note that as shown in the proof of Lemma \ref{RS}, we have the same zero-free region given in \eqref{zetaKzerofree} holds for $L(s, \chi_4)$. Thus, similar to \cite[Lemma 6]{HT86}, we see that under the same conditions as for the $\zeta(s)$ case, we have
\begin{align}
\label{Lambdapartialsum1}
\begin{split}
  \sum_{n \leq y}\frac {\Lambda(n)\chi(n)}{n^s}
=& O\Big(\frac 1{\beta}(1+y^{\beta}\exp(-(\log y)^{\varepsilon/2}))\Big ).
\end{split}
\end{align}    

  Using \eqref{Lambdapartialsum} and \eqref{Lambdapartialsum1} allows us to see, in the same manner as the derivation of the Corollary to \cite[Lemma 6]{HT86}, that under the same condition as above, we have  
\begin{align}
\label{lambdadiffest}
\begin{split}
   \sum_{n \leq y}\frac {\Lambda(n)(1+\chi(n))}{n^{1-\beta}}(1-\cos(t\log n))=\frac {y^{\beta}}{\beta}(1-\beta\frac {\eta}{\sqrt{\beta^2+t^2}})+O_{\varepsilon}\Big(\frac 1{\beta}(1+y^{\beta}\exp(-(\log y)^{\varepsilon/2}))\Big),
\end{split}
\end{align}  
  where $\eta=t\log y -\arctan(t/\beta)$.

  Note moreover that by \cite[Lemma 5]{HT86}, we have uniformly for $x \geq y \geq 2$, 
\begin{align}
\label{Lambdapartialsum2}
\begin{split}
  \frac 1{\log y}\sum_{\nu \geq 2}\sum_{p^{\nu} \leq y}\frac {\log p\chi(p^{\nu})}{p^{\nu \alpha_G}}, \ \  \frac 1{\log y}\sum_{\nu \geq 2}\sum_{p^{\nu} \leq y}\frac {\log p}{p^{\nu \alpha_G}} \ll 1+y^{1/2-\alpha_G}. 
\end{split}
\end{align}   

 Observe also that for $1/2+\varepsilon_1 \leq \alpha_G \leq 1-4/\log y$ and $y$ large enough, we have $u \gg y^{1-\alpha_G}$ by \eqref{ypowerupper}. 
We now apply \eqref{lambdadiffest} with $\beta=1-\alpha_G$ and \eqref{Lambdapartialsum2} to proceed as in the proof of \cite[Lemma 8, (ii)]{HT86} 
to estimate the right-hand side expression in \eqref{Lratiotlarge} to see that the estimation given in \eqref{Lrationtlarge} is valid. This completes the proof of the lemma. 
\end{proof}

   Our last result in this section provides an estimation for the difference of the $\Psi_G$ function, which is an analogue to \cite[Lemma 9]{HT86}.
\begin{lemma}
\label{lemphidiffbound}
  With the notation as above. There exists a constant $d_0>0$ such that uniformly for $x \geq y \geq 2, 0<\varepsilon <1, 1 \leq z \leq Y(\lambda)$, we have
\begin{align}
\label{phidiffest}
  \Psi_G(x+x/z,y)-\Psi_G(x,y) \ll _{\varepsilon} x^{\alpha_G}H(\alpha_G,G;y)(\frac 1z+e^{-d_0u}).
\end{align}
\end{lemma}  
\begin{proof}
 One checks readily that the left-hand side expression above is
\begin{align*}
  \ll \sum_{\substack{n \geq 1 \\ P(n) \leq y}}r(n)\exp \Big (-\frac 12 (z\log (\frac {x}{n}))^2 \Big ).
\end{align*}
  We then proceed similar to the proof of \cite[Lemma 9]{HT86} upon using part (ii) of Lemma \ref{RS3} and \eqref{Hbound} to see that the estimation given in \eqref{phidiffest} is valid. This completes the proof of the lemma. 
\end{proof}

\subsection{Completion of the proof}
  
  We may assume that $y$ is large enough so that by \eqref{alphalowerbound} and \eqref{equ: Saddle point approx in intro1} we have $1/2+\varepsilon_1 \leq \alpha_G \leq 1-4/\log y$ for some fixed constant $\varepsilon_1>0$. We set $T=Y(\frac 14)$ and apply Perron's formula similar to that given in \cite[(4.5)]{MVa1} to see that
\begin{align}
\label{sumhlambda}
\begin{split}
\sum_{\substack{n \leq x \\ P(n) \leq y}} r(n) =&  \frac{4}{2 \pi i} \int_{\alpha_G - i T}^{\alpha_G + iT} H(s,G;y) x^s \frac{ds}{s} + O \Biggl(  x^{\alpha_G}\sum_{\substack{n \geq 1 \\ P(n) \leq y}}\frac {r(n)}{n^{\alpha_G}}\min \biggl\{ 1, \frac{1}{T|\log (x/n)|}\biggl\} \Biggr) \\
=& \frac{4}{2 \pi i} \int_{\alpha_G - i T}^{\alpha_G + iT} H(s,G;y) x^s \frac{ds}{s} + O \Biggl(  \frac {x^{\alpha_G}H(\alpha_G, G;y)}{\sqrt{T}}+ x^{\alpha_G}\sum_{\substack{ P(n) \leq y \\ |\log (x/n)| \leq T^{-1/2}}}\frac {r(n)}{n^{\alpha_G}}\Biggr).
\end{split} 
\end{align} 
  
  Note that we have for some constant $d_1>0$,
\begin{align}
\label{sumnaroundx}
\begin{split}
 x^{\alpha_G}\sum_{\substack{ P(n) \leq y \\ |\log (x/n)| \leq T^{-1/2}}}\frac {r(n)}{n^{\alpha_G}} \ll \Psi_G(x+\frac {d_1x}{\sqrt{T}},y)- \Psi_G(x-\frac {d_1x}{\sqrt{T}},y) \ll_{\varepsilon} x^{\alpha_G}H(\alpha_G, G;y)(\frac 1{\sqrt{T}}+e^{-d_0u}),
\end{split} 
\end{align} 
  where the last estimation above follows from Lemma \ref{lemphidiffbound} and where $d_0>0$ is the same constant given there. 

  We apply \eqref{sumnaroundx} in \eqref{sumhlambda} to see that for any $\varepsilon>0$, we have
\begin{align}
\label{sumhlambda1}
\begin{split}
\sum_{\substack{n \leq x \\ P(n) \leq y}} r(n) = \frac{4}{2 \pi i} \int_{\alpha_G - i T}^{\alpha_G + iT} H(s,G;y) x^s \frac{ds}{s} + O (x^{\alpha_G}H(\alpha_G, G;y)(\frac 1{\sqrt{T}}+e^{-d_0u})).
\end{split} 
\end{align}   
 
  We now follow the treatment in the proof of \cite[Lemma 10]{HT86} to apply part (i) of Lemma \ref{RS3} see that for some constant $b_7>0$, 
\begin{align}
\label{sumhlambda2}
\begin{split}
 \int_{\alpha_G + i /\log y}^{\alpha_G + iT} H(s,G;y) x^s \frac{ds}{s} + \int_{\alpha_G - i /\log y}^{\alpha_G - iT} H(s,G;y) x^s \frac{ds}{s} \ll x^{\alpha_G}H(\alpha_G,G;y)\exp(-b_7u(\log 2u)^{-2})\log T.
\end{split} 
\end{align}   

  We next note that similar to the proof of Lemma \ref{lemphibound}, we have uniformly for $1 \leq u \leq y/(8\log y), \alpha_G=\alpha_G(x,y) \geq 1/2+\varepsilon_1$ that
\begin{align}
\label{phikasymp1}
  \phi_4(\alpha_G+it,G;y) \ll u(\log y)^4. 
\end{align}
  We now proceed as in the proof of \cite[Lemma 11]{HT86} to see by \eqref{phikasymp1}, Lemma \ref{lemphibound} that
\begin{align}
\label{sumhlambda3}
\begin{split}
 \frac{1}{2 \pi i}\int_{\alpha_G - i /\log y}^{\alpha_G + i/\log y} H(s,G;y) x^s \frac{ds}{s} = \frac {x^{\alpha_G}H(\alpha_G,G;y)}{\alpha_G\sqrt{2\pi \phi_2(\alpha_G, G;y)}}(1+O(\frac 1u)).
\end{split} 
\end{align}   

  We gather \eqref{sumhlambda1}, \eqref{sumhlambda2} and \eqref{sumhlambda3} to see that
\begin{align}
\label{sumhlambda4}
\begin{split}
\sum_{\substack{n \leq x \\ P(n) \leq y}} r(n) =&  \frac {4x^{\alpha_G}H(\alpha_G,G;y)}{\alpha_G\sqrt{2\pi \phi_2(\alpha_G, G; y)}}\Big (1+ O(E) \Big ),
\end{split} 
\end{align}   
 where
\begin{align}
\label{sumhlambda5}
\begin{split}
E
=&  \frac {\alpha_G\sqrt{\phi_2(\alpha_G, G; y)}}{\sqrt{T}}+ \alpha_G\sqrt{\phi_2(\alpha_G, G; y)}e^{-d_0u}+\alpha_G\sqrt{\phi_2(\alpha_G, G; y)}\exp(-b_7u(\log 2u)^{-2})\log T+\frac 1u .
\end{split} 
\end{align}   

  Note that we have $T=Y(\frac 14)=\exp((\log y)^{5/4}),  (\log \log 2y)^2 \leq u \leq y/(8\log y), 1/2+\varepsilon_1 \leq \alpha_G \leq 1-4/\log y$. Also, we have by Lemma \ref{lemphibound} that 
$\sqrt{\phi_2(\alpha_G,G;y)} \ll u^{1/2}\log y$. We apply these estimations into \eqref{sumhlambda4} and \eqref{sumhlambda5} to see that the expression given in \eqref{psiasymp} is valid. This completes the proof of Theorem \ref{thm: psiasymp}.

\vspace*{.5cm}

\noindent{\bf Acknowledgments.} P. G. is supported in part by NSFC grant 12471003.

\bibliography{biblio}
\bibliographystyle{amsxport}

\end{document}